\documentclass[12pt]{amsart}
\usepackage{geometry}                
\usepackage{graphicx}
\usepackage{amssymb}
\usepackage{epstopdf}
\newtheorem{theorem}{Theorem}

\newtheorem{lemma}{Lemma}

\newtheorem{proposition}{Proposition}

\numberwithin{equation}{section}

\title{Implications of subconvexity bounds for the moments of $\zeta(s)$}

\author{Kevin Smith}

\begin{document}
\maketitle

\begin{abstract}
It is well-known that upper bounds for moments of the Riemann zeta function $\zeta(s)$ have implications for subconvexity bounds. In this paper we explore some implications in the opposite direction using functional analysis in the right-half of the critical strip. The main results characterise potential transitions in the behaviour of the moments. 
\end{abstract}
\section{Introduction}
Since Montgomery's work \cite{M} on the pair-correlation of the zeros of the Riemann zeta function $\zeta(s)$, in particular the work of Katz and Sarnak \cite{KSar}, Keating and Snaith \cite{KS}, Conrey, Farmer, Keating, Rubinstein and Snaith \cite{CFKRS} and Farmer, Gonek and Hughes \cite{FGH}, it is now widely believed that zero distributions, moments and bounds in families of $L$-functions approximating critical lines and critical points correspond to those of the characteristic polynomials of large random matrices in one of the classical compact groups. For the moments of $\zeta(s)$,
\begin{eqnarray}
M_k(\sigma,T)=\int_0^T    \left|\zeta\left(\sigma+it\right)\right|^{2k} dt,\nonumber
\end{eqnarray}
it was conjectured (assuming the Riemann hypothesis) by Keating and Snaith \cite{KS} by analogy with the unitary case, and by Conrey and Ghosh \cite{CGh} and Conrey and Gonek \cite{CG} using conjectures on additive divisor sums, that 
 \begin{eqnarray}
 M_k(1/2,T)\sim a_kg_k T(\log T)^{k^2}
 \end{eqnarray} 
 where
\begin{eqnarray}\label{as}
a_k=\prod_{p\textrm{ prime}}\left(1- \frac{1}{p}\right)^{(k-1)^2}\sum_{n= 0}^{\infty} {k-1\choose n}^2p^{-n}\hspace{0.4cm}\textrm{and    }\hspace{0.4cm}g_k=\frac{G^2(k+1)}{G(2k+1)}
\end{eqnarray}
in which $G(z)$ is Barnes' $G$-function. In particular, the work of Keating and Snaith and Gonek, Hughes and Keating \cite{GHK} predicts that the factor $g_k(\log T)^{k^2}$ arises from the influence of the zeros and their distribution on the critical line $\sigma=1/2$, while the constant factor $a_k$ arises from the influence of the primes. More recently, building on work of Soundararajan \cite{S}, Harper \cite{H} has shown that the Riemann hypothesis alone implies that
\begin{eqnarray}\label{conj}
M_k(1/2,T)\ll_k T(\log T)^{k^2}\hspace{1cm}(k>0).
\end{eqnarray} 
Lower bounds for $M_k(1/2,T)$ of the correct order of magnitude for all $k\geq 1$ were computed by  Radziwiłł and Soundararajan \cite{RS} and for all $k>0$ by Heap and Soundararajan \cite{HS}. However, the upper bound in (\ref{conj}) is known unconditionally only when $k\leq 2$. The cases $k=1$ and $2$ are classical results of Hardy and Littlewood \cite{HL} and Ingham \cite{I}, respectively. 
The difficulties involved in computing $M_k(\sigma,T)$ increase not only with increasing $k$ on the critical line, but also as $\sigma$ decreases toward $1/2$ for each fixed $k\in\mathbb{N}$. However, in the region $1/2<\sigma<1$ where zeros are rarer (of course, the Riemann hypothesis predicts there are none), it is expected that $M_k(\sigma,T)$ is completely determined by the influence of the primes. Indeed, this is a consequence of the Lindel\"of hypothesis which we discuss below.

Denoting by $d_k(n)$ the number of ways of writing $n$ as a product of $k$ factors,  $M_k(\sigma,T)$ obeys a convexity principle as a function of $\sigma$, one consequence of which is that the statements 
\begin{eqnarray}\label{asy1}
M_k(\sigma_k,T) \ll T^{1+\epsilon}\hspace{1cm}(\sigma_k\geq 1/2),
\end{eqnarray}
\begin{eqnarray}\label{asy2}
M_k(\sigma,T)\ll T \hspace{1cm}(1/2\leq\sigma_k<\sigma<1)
\end{eqnarray}
and
\begin{eqnarray}\label{asy3}
\lim_{T\rightarrow\infty}\frac{M_k(\sigma,T)}{T} =\sum_{n=1}^{\infty}\frac{d^2_k(n)}{n^{2\sigma}} \hspace{1cm} (1/2\leq\sigma_k<\sigma<1)
\end{eqnarray}
are equivalent (Titchmarsh \cite{Titch}, Section 7.9). Another consequence is that for every $k\in\mathbb{N}$ there is a $\sigma_k<1$ with the above property. General results of this type are given by Iv\'ic (\cite{IV}, Section 8.5). However, beyond the theorems of Hardy and Littlewood and Ingham mentioned above which yield  $\sigma_1=\sigma_2=1/2$, currently the best known estimates are $\sigma_3\leq 7/12$ (due to Iv\'ic cited above), $\sigma_4\leq 5/8$ and $M_3(1/2,T)\ll_{\epsilon} T^{5/4+\epsilon}$. The latter two results are consequences of Heath-Brown's  twelfth moment estimate \cite{HB}. 

Subconvexity bounds are described by the continuous non-negative non-increasing convex function
\begin{eqnarray}
\mu(\sigma)=\inf\{\mu:\zeta(\sigma+it)\ll t^{\mu}\}.\nonumber
\end{eqnarray}
Here, the Lindel\"of hypothesis asserts that $\mu(1/2)=0$ (currently the best known bound is $\mu(1/2) < 13/84$ due to Bourgain \cite{Bo3}). 
In terms of moments, the Lindel\"of hypothesis is equivalent to the statement that $M_k(1/2,T) \ll_{k,\epsilon} T^{1+\epsilon}$ for all $k\in\mathbb{N}$, i.e. 
\begin{eqnarray}\label{asy4}
\mu(1/2)=0\Leftrightarrow \sigma_k=1/2\hspace{0.5cm} (k\in\mathbb{N})
\end{eqnarray}
due to Hardy and Littlewood \cite{HL2} and, since (\ref{asy1}), (\ref{asy2}) and (\ref{asy3}) are equivalent, it is evident that the Lindel\"of hypothesis depends only on the moments in the region $\sigma>1/2$.
In fact, it is a consequence of quite general properties of $\zeta(s)$ that the implication (\ref{asy4}) has a ``pointwise'' analogue in one direction at least, i.e.
\begin{eqnarray}\label{imp1}
\sigma_k=1/2 \Rightarrow \mu(1/2)<1/2k
\end{eqnarray}
(for instance, see Iv\'ic (\cite{IV}, Section 8.1) or Heath-Brown's notes (\cite{Titch}, Chapter 7). Yet, beyond that which can be derived as a consequence of convexity, i.e. results of the type given by Iv\'ic (\cite{IV}, Section 8.5), the implications of
\begin{eqnarray}\label{mustate}
\mu(1/2)<1/2k
\end{eqnarray}
for the numbers $\sigma_k$ are unclear. In particular, the converse statement of (\ref{imp1}) is not known. Similarly, despite that the inverse statement $\sigma_k>1/2$ certainly implies that $\mu(1/2)>0$, in other words
\begin{eqnarray}\label{inverse}
\limsup_{T\rightarrow\infty}\frac{M_k(\sigma,T)}{T}=\infty\hspace{1cm}(1/2\leq \sigma< \sigma_k),
\end{eqnarray}
it is not apparent what lower bound for $\mu(1/2)$ could be given if $\sigma_k>1/2$ were true for some $k\in\mathbb{N}$.

In the absence of further information in this direction, in this paper we assume (\ref{mustate}) and establish some results of a probabilistic character. The main idea in our work is that (\ref{mustate}) permits us to perform a sort of ``Fourier analysis'' that leads to particular insights regarding those values of $1/2<\sigma<1$ for which $M_k(\sigma,T)\ll_k T$, or not. 
These are contained in theorems  \ref{zetacns} and \ref{thzeta} below.

We denote by
\begin{eqnarray}\label{densitydef}
|S|=\lim_{T\rightarrow\infty}\frac{\textrm{meas}\left(S\cap [0,T]   \right)}{T}\nonumber
\end{eqnarray}
the natural density of Lebesgue measurable subsets $S\subseteq\mathbb{R}^{+}$ (if $|S|=0$ we say that $S$ is null) and 
$\textrm{Pr}\left(|g|\geq x\right)=|X|$ where $X\subseteq \mathbb{R}^{+}$ is the subset on which $|g|\geq x$. To begin,  we observe that the trivial bound 
\begin{eqnarray}\label{holderbound}
\frac{\left|\int_{0}^Tf^2gdt\right|}{\int_{0}^T|f|^2dt}\leq\textrm{ess sup}_{[0,T]}|g|\nonumber
\end{eqnarray}
may be significantly improved when $T$ is large if $f$ is not concentrated on a null set, by which we mean the following. 

\begin{proposition}[Concentration on a null set]\label{holder} The following statements are equivalent. If $f$ satisfies them, we say that $f$ is not concentrated on a null set (not CNS).
\begin{itemize}
\item[(\emph{a})]{If $g$ is bounded then
\begin{eqnarray}\label{prob}
\limsup_{T\rightarrow\infty}\frac{\left|\int_{0}^Tf^2gdt\right|}{\int_{0}^T|f|^2dt}\leq \sup\{x:\Pr(|g|\geq x)>0\}.
\end{eqnarray}
}
\item[(\emph{b})]{If $S$ is null then 
\begin{eqnarray}\label{meas}
\int_{S\cap[0,T]}|f|^2dt=o\left(\int_{0}^T|f|^2dt\right)\hspace{1cm}(T\rightarrow\infty).
\end{eqnarray}
}
\end{itemize}
Moreover, if $\int_{0}^T|f|^2dt\gg T$, we may include 
\begin{itemize}
\item[(\emph{c})]{If $S$ is null then for every $\epsilon>0$ there is a bounded function $f_{N,S}$ $(N=N(\epsilon))$ such that 
\begin{eqnarray}\label{ot}
\limsup_{T\rightarrow\infty}\frac{\int_{S\cap[0,T]}|f-f_{N,S}|^2dt}{\int_{0}^T|f|^2dt}< \epsilon.
\end{eqnarray}
}
\end{itemize}
If also $\int_{0}^T|f|^2dt\ll T$ then condition \emph{(}c\emph{)} may be replaced with 
\begin{itemize}
\item[(\emph{d})]{For every $\epsilon>0$ there is a bounded function $f_{N}$ $(N=N(\epsilon))$ such that 
\begin{eqnarray}\label{sup}
\limsup_{T\rightarrow\infty}\frac{1}{T}\int_{0}^{T}|f-f_{N}|^2dt< \epsilon.
\end{eqnarray}
}
\end{itemize}
\end{proposition}
 
\begin{theorem}\label{zetacns}
Assume $\mu(1/2)<1/2k$, $1/2<\sigma<1$ and $\sigma\neq \sigma_k$. Then $\zeta^k(\sigma+it)$ is not CNS if and only if $\sigma_k<\sigma<1$.
\end{theorem}

Thus, Theorem \ref{zetacns} shows that if $\mu(1/2)<1/2k$ and $\sigma_k>1/2$ then the density 
\begin{eqnarray}\label{density}
\frac{|\zeta(\sigma+it)|^{2k}}{M_k(\sigma,T)}\hspace{1cm}(t\in[0,T],k\in\mathbb{N})
\end{eqnarray}
must undergo a critical transition at the point $\sigma_k$ in the sense that (\ref{meas}) holds with $f(t)=|\zeta(\sigma+it)|^{k}$ for $\sigma_k<\sigma<1$ and fails for $1/2<\sigma< \sigma_k$. Since we know that $\mu(1/2)<1/6$ and $1/2\leq \sigma_3\leq 7/12$, this suggests an alternative way to try to rule out the possibility that $\sigma_3>1/2$. Theorem \ref{zetacns} arises as a special case of the more general statement Proposition \ref{zo} given in Section \ref{general}, which captures the influence of a broader class of locally integrable functions being concentrated on a null set. It is here that functional analysis plays a key role.

One explanation for the apparent disparity between bounds for $\mu(1/2)$ and bounds for $\sigma_k$, and indeed our ability to compute the moments, is that such atypical large values of $|\zeta(\sigma+it)|^{k}$ (if they exist) must occur ``near'' the imaginary ordinates of zeros away from the critical line (if they exist). This is clear from, for example, the proof of Proposition \ref{prop4} below. Yet, the available estimates for the density of zeros decay (as a function of $\sigma$) too slowly for us to draw the inverse or converse conclusions from them alone. Here several different types of estimate for the number $N(\sigma,T)$ of zeros $\rho=\beta+i\gamma$ with $\beta\geq\sigma$ and $|\gamma|\leq T$ are available (\cite{Titch}, Section 9.15), but these typically involve functions $\alpha(\sigma)$ and $\beta(\sigma)$ such that 
\begin{eqnarray}\label{zbound}
N(\sigma,T)\ll T^{\alpha(\sigma)}(\log T)^{\beta(\sigma)}
\end{eqnarray} in which $\alpha(\sigma)$ decays slowly in the range $\sigma>1/2$ with $\alpha(1/2)=1$. Nonetheless, these zero density estimates are sufficient to establish Proposition \ref{prop4} below. Here we set $\theta_{P_N}=\arg P_N$ where
\begin{eqnarray}\label{Pdef}
P_N(\sigma+it)=\exp\left(\sum_{2\leq n\leq N^2}\frac{\Lambda_N(n)}{n^{\sigma+it}\log n}\right),
\end{eqnarray}
\begin{eqnarray}\label{kjj34322}
\Lambda_N(n)&=&\left\{
                \begin{array}{ll}
                  \Lambda(n) \hspace{4cm}(n\leq N)  \\
               \Lambda(n)\left(2-\frac{\log n}{\log N}\right)\hspace{1.9cm}(\textrm{otherwise})
                \end{array}
              \right.\nonumber
\end{eqnarray}
in which $\Lambda(n)$ is the Von-Mangoldt function
\begin{eqnarray}
\Lambda(n)&=&\left\{
                \begin{array}{ll}
                  \log p \hspace{1.3cm}(n=p^m, p \textrm{ prime}, m\in\mathbb{N})  \\
               0\hspace{1.9cm}(\textrm{otherwise}).
                \end{array}
              \right.\nonumber
\end{eqnarray}\\
We also set  $Z_N=\zeta/P_N$ and
\begin{eqnarray}
\theta_{Z_N}=\arg Z_N\nonumber
\end{eqnarray}
so $\theta_{Z_N}=\theta_{\zeta}-\theta_{P_N}$ $\pmod {(-\pi,\pi]}$ and 
\begin{eqnarray}
|\zeta|^{2k} e^{2ik\theta_{Z_N}}=|P_N|^{2k}Z_N^{2k}=\zeta^{2k} e^{-2ik\theta_{P_N}}\hspace{1cm}(k\in\mathbb{N}).\nonumber
\end{eqnarray}

\begin{proposition}\label{prop4} For every $1/2<\sigma<1$ and $\epsilon>0$
\begin{eqnarray}
\textrm{\emph{Pr}}\left(\left|\theta_{ Z_N}(\sigma+it)\pmod {(-\pi,\pi]}\right|\geq \epsilon\right)=0\hspace{1cm}(N>N(\epsilon)). \nonumber
\end{eqnarray}
\end{proposition}
Proposition \ref{prop4} plays a key role in our proofs because it implies that $f=e^{ik\theta_{\zeta}}$ satisfies (\ref{sup}) with $f_N=e^{ik\theta_{P_N}}$ for $1/2<\sigma<1$ and $k\in\mathbb{N}$.

Our second theorem demonstrates a ``zero-one'' law for the density (\ref{density}) in the range $1/2<\sigma<1$.

\begin{theorem}\label{thzeta}Assume $\mu(1/2)<1/2k$. For every $\epsilon>0$
\begin{eqnarray}
\limsup_{T\rightarrow\infty}\frac{\int_{0}^T|\zeta(\sigma+it)|^{2k}\sin^2(k\theta_{Z_N}(\sigma+it))dt}{M_k(\sigma,T)}
\left\{
                \begin{array}{ll}
               =1/2\hspace{0.4cm}(N\in\mathbb{N},1/2<\sigma< \sigma_k)\\
               <\epsilon \hspace{0.78cm}\left(N>N(\epsilon), \sigma_k<\sigma<1\right).
                \end{array}
              \right.\nonumber
\end{eqnarray}
\end{theorem}

One side of this law states that, if $\sigma_k>1/2$, then for every $1/2<\sigma< \sigma_k$ and $N\in\mathbb{N}$ the distribution of the angles 
$k\theta_{ Z_N}(\sigma+it)$ $\pmod {(-\pi,\pi]}$ with respect to the density (\ref{density}) is consistent with that of a continuous random variable uniformly distributed on $(-\pi,\pi]$ with respect to the density (\ref{density}). 
On the other hand, the other side of this law states that the expected value of those angles with respect to the density (\ref{density}) converges to $0$ as $N\rightarrow\infty$  for every $\sigma_k<\sigma< 1$. Thus we obtain a highly structured consequence of the event  $\sigma_k>1/2$ conditionally on $\mu(1/2)<1/2k$. In light of Proposition \ref{prop4} above, Theorem \ref{thzeta} lends support to the hypothesis that it is more natural for the events $\mu(1/2)<1/2k$ and $\sigma_k>1/2$ to be mutually exclusive, i.e. that $\mu(1/2)<1/2k\Rightarrow \sigma_k=1/2$, which would be a favourable outcome as it implies that $\sigma_3=1/2$.

The proofs of theorems \ref{zetacns} and \ref{thzeta} are postponed until Section \ref{zeta} because several details of the arguments are quite general and the functional analysis in Section \ref{general} handles them in a clearer way. The proofs of propositions \ref{holder} and \ref{prop4} are given in Section \ref{props}.

\section{Preliminaries. Almost-periodicity}\label{general}
Throughout this paper we write $\zeta=\zeta(\sigma+it)$ for any fixed $\sigma\in(1/2,1)$ unless otherwise specified, and $f,g\in L_{2,\textrm{loc}}(\mathbb{R}^{+})$. Writing 
\begin{eqnarray} \label{ip}
\langle f, g\rangle=\lim_{T\rightarrow\infty}\frac{1}{T}\int_{0}^{T}f(t)\overline{g(t)}dt
\end{eqnarray}
and $\|f\|=\sqrt{\langle f,f\rangle}$ when the limits exists, it is clear from (\ref{sup}) and the reverse triangle inequality that if the $\|f_N\|$ exist then $\|f\|=\lim_{N\rightarrow \infty}\|f_{N}\|$. For example, this is the case when the $f_{N}$ in  (\ref{sup}) are Bohr almost-periodic functions, that is, the $f_N$ have a (necessarily bounded and continuous) uniformly convergent generalised Fourier series
\begin{eqnarray}\label{Bohr}
g= \sum_{\lambda \in \mathbb{R}} c_{\lambda} e_{\lambda}
\end{eqnarray}
where $e_{\lambda}(t)=e^{i\lambda t}$, in which case the Fourier coefficients $c_{\lambda}= \langle g,e_{\lambda}\rangle$ are necessarily square-summable and $\|f_N\|^2=\sum_{\lambda\in\mathbb{R}}|\langle f_N,e_{\lambda}\rangle|^2$. In this case, the $f$ in (\ref{sup}) belongs to the closure of the space of Bohr functions in the seminorm $\|\cdot\|$. In other words, the equivalence class $\phi+f$ for $\|\phi\|=0$ belongs to the Hilbert space $B^2=B^2(\mathbb{R}^{+})$  of Besicovitch almost-periodic function on $\mathbb{R}^{+}$, in which case we write $f_N\rightarrow f$ in $B^2$ as $N\rightarrow\infty$. For $f,g\in B^2$ the inner product  (\ref{ip}) exists and the Parseval relation\footnote{The coefficients $\langle f,e_{\lambda}\rangle$ are necessarily non-zero for at most countably many $\lambda\in\mathbb{R}$.}
\begin{eqnarray}\label{parseval}
\langle f,g\rangle=\sum_{\lambda\in\mathbb{R}}\langle f,e_{\lambda}\rangle \overline{\langle g,e_{\lambda}\rangle}
\end{eqnarray}
holds. Conversely, if $f\in B^2$ then for instance we may take $f_N$ to be a partial sum
\begin{eqnarray}\label{partial}
f_{N}= \sum_{n\leq N} \langle f,e_{\lambda_n}\rangle e_{\lambda_n}
\end{eqnarray}
and verify (\ref{sup}) for this sequence using Bessel's inequality. Thus, if $f\in B^2$ then $f$ is not CNS. Reasoning along similar lines, we have the following pointwise equivalence. 

\begin{proposition}\label{sig}For each $k\in\mathbb{N}$, $\sigma_k=1/2\Leftrightarrow\zeta^k\in B^2$.
\end{proposition}
\noindent 
Therefore, since $\sigma_1=\sigma_2=1/2$, the functions  $\zeta, \zeta^2\in B^2$ are not CNS.

Although it is not known that $\sigma_k=1/2$ for $k>2$, our next proposition permits us to compute the Fourier coefficient of $\zeta^k$ in any case. 
\begin{proposition}\label{lem1}For each $k\in\mathbb{N}$, if $\mu(1/2)<1/k$ then
\begin{eqnarray}
\langle \zeta^k,e_{\lambda}\rangle = \left\{
                \begin{array}{ll}
                  \frac{d_{k}(n)}{n^{\sigma}} \hspace{1.25cm}(\lambda=-\log n \hspace{0.3cm}(n\in\mathbb{N}))  \\
               0\hspace{1.83cm}(\textrm{\emph{otherwise}}).
                \end{array}
              \right.\nonumber
              \end{eqnarray}
\end{proposition}
\noindent Thus, by Proposition \ref{lem1} and the fact that $\mu(1/2)<1/6$, there is a broader class of ``almost-periodic'' functions to which the functions $\zeta^k$ ($1\leq k\leq 6$) belong, which we now introduce.

\begin{lemma}\label{mainlemma}Let $E=\{e_{\lambda}:\lambda\in\mathbb{R}\}$ and consider the set of all $f$ satisfying 
\begin{eqnarray}\label{coeffbound}
\|f\|_{A^2}^{2}=\sum_{\lambda\in \mathbb{R}}|\langle f,e_{\lambda}\rangle |^2<\infty.\nonumber
\end{eqnarray}
Then the set $A^2=A^2(\mathbb{R}^+)$ of equivalence classes $\phi+f$ with $\phi\in E^{\perp}$ is a Hilbert space with inner product
\begin{eqnarray}\label{innerprodH1}
\langle f,g\rangle_{A^2} =\sum_{\lambda\in \mathbb{R}}\langle f,e_{\lambda}\rangle \overline{ \langle g,e_{\lambda}\rangle }
\end{eqnarray}
and we have 
\begin{eqnarray}\label{decomp}A^2=E^{\perp}\oplus B^2.
\end{eqnarray}
\end{lemma}

It is clear from (\ref{decomp}) that if $f\in A^2$ then $f\in B^2\Leftrightarrow\|\phi\|= 0$. However, the space $A^2$ does contain many more functions than $B^2$ and it is easily checked that  there exist $\phi\in E^{\perp}$ for which $0<\|\phi\|\leq\infty$, for example
\begin{eqnarray} \label{functions}
\left\|\frac{\zeta\left(1/2+i\cdot\right)}{\log^{1/2} (\cdot+1)}\right\|=1\hspace{0.2cm}  \textrm{  and }\hspace{0.2cm}  \left\|\frac{\zeta\left(1/2+i\cdot\right)}{\log^{1/3} (\cdot+1)}\right\|=\infty.\nonumber
\end{eqnarray}
Nonetheless, if $f\in A^2$ has certain properties that we assume in Proposition \ref{zo} below, then we may conclude that there are precisely two possibilities: either $f\in B^2$, or $f$ is CNS. Before stating the result, we note that $B^2$ has a particular group of unitary operators acting on it. This is the content of Lemma \ref{ugroup}.

\begin{lemma}\label{ugroup}
The set $U=\{u\in B^2:|u|=1\}$  is a multiplicative subgroup of $B^2$. In particular, if $u\in U$ then $u^{j}\in U$ $(j\in\mathbb{Z})$. Moreover, if $u_{\lambda}=ue_{\lambda }$ then the set $\{u_{\lambda }: \lambda\in\mathbb{R}\}$ is an orthonormal basis for $B^2$. Equivalently, the map $f\rightarrow uf$ $(u\in U)$ is a unitary operator on $B^2$. 
\end{lemma}

\noindent Propositions \ref{sig}, \ref{lem1} and lemmas \ref{mainlemma}, \ref{ugroup} are proved in Section \ref{auxs}.

The main result of this section is Proposition \ref{zo}, where we write $\theta_f=\arg f$.

\begin{proposition}\label{zo}Let $g\in B^2$, $\langle f,e_{\lambda}\rangle=\langle g,e_{\lambda}\rangle$ and $\langle f^2,e_{\lambda}\rangle=\langle g^2,e_{\lambda}\rangle$ $(\lambda\in\mathbb{R})$, so that $f,f^2\in A^2$. Also  let $e^{i\theta_f},e^{i\theta_g}\in U$, $\langle e^{i\theta_f},e_{\lambda}\rangle=\langle e^{i\theta_g},e_{\lambda}\rangle$ $(\lambda\in\mathbb{R})$, and let $e^{i\theta_{g_N}}$ be a sequence of Bohr functions such that $e^{i\theta_{g_N}}\rightarrow e^{i\theta_g}$ in $B^2$ as $N\rightarrow\infty$. Then we have the following.
\begin{itemize}
\item[(\emph{a})]{$f\in B^2$ if and only if $f$ is not CNS.}
\item[(\emph{b})]{ If the maps $h\mapsto e^{i\theta_f}h$ and $h\mapsto e^{i\theta_g}h$ are bounded operators on $A^2$, then $f\in B^2$.}
\end{itemize}
\end{proposition}

\begin{proof}[Proof of Proposition \ref{zo}]As we have seen, if $f\in B^2$ then $f$ is not CNS, so we now assume that $f$ is not CNS and prove the converse statement. Using the identity $1-\cos(2x)=2\sin^2(x)$  we have 
\begin{eqnarray}\label{ebound}
0 &\leq &\frac{1}{T}\int_{0}^T|f|^2dt-\Re\frac{1}{T} \int_{0}^Tf^2e^{-2i\theta_{g_N}}dt \nonumber\\ &=&\frac{2}{T}\int_{0}^T|f|^2\sin^2(\theta_f-\theta_{g_N})dt\nonumber \\&\leq & \frac{2\sup\{x:\Pr(\sin^2(\theta_f-\theta_{g_N})\geq x)>0\}+2\epsilon}{T}\int_{0}^T|f|^2dt
\end{eqnarray}
for every $\epsilon>0$ and  $T\geq T(\epsilon)$ by (\ref{prob}). Since $\lim_{N\rightarrow\infty}\|e^{i\theta_f}-e^{i\theta_{g_N}}\|=0$, there is also an $N(\epsilon)$ such that $\Pr(\sin^2(\theta_f-\theta_{g_N})\geq \epsilon)=0$ for $N\geq N(\epsilon)$, so (\ref{ebound}) shows that
\begin{eqnarray}
0 \leq \frac{1}{T}\int_{0}^T|f|^2dt-\Re\frac{1}{T} \int_{0}^Tf^2e^{-2i\theta_{g_N}}dt\leq \frac{4\epsilon}{T}\int_{0}^T|f|^2dt\nonumber
\end{eqnarray}
for $N\geq N(\epsilon)$ and $T\geq T(\epsilon)$, which gives
\begin{eqnarray}\label{ebound2}
\Re\frac{1}{T} \int_{0}^Tf^2e^{-2i\theta_{g_N}}dt\leq \frac{1}{T}\int_{0}^T|f|^2dt\leq \frac{1}{1-4\epsilon}\Re\frac{1}{T} \int_{0}^Tf^2e^{-2i\theta_{g_N}}dt
\end{eqnarray}
for $0<\epsilon<1/4$. Since $e^{-2i\theta_{g_N}}$ is a Bohr function it has a uniformly convergent Fourier series of the form (\ref{Bohr}), so the map $f\mapsto e^{-2i\theta_{g_N}} f$ is a bounded operator on $A^2$. Therefore, since $f^2\in A^2$, the limits 
\begin{eqnarray}\label{ebound3}
\langle f^2, e^{2i\theta_{g_N}}\rangle=\lim_{T\rightarrow\infty}\frac{1}{T} \int_{0}^Tf^2e^{-2i\theta_{g_N}}dt
\end{eqnarray}
exist and from (\ref{ebound2}) we see that the second moment  $\|f\|^2$ exists and 
\begin{eqnarray}\label{iprod2}
\|f\|^2=\lim_{N\rightarrow\infty}\Re\langle f^2, e^{2i\theta_{g_N}}\rangle.
\end{eqnarray} 
Moreover, since $\langle f^2,e_{\lambda}\rangle=\langle g^2,e_{\lambda}\rangle$ and $e^{2i\theta_{g_N}}$ is a Bohr function, we have
\begin{eqnarray}\label{ser}
\langle f^2,e^{2i\theta_{g_N}}\rangle=\left\langle g^2,e^{2i\theta_{g_N}}\right\rangle=\left\langle g,e^{2i\theta_{g_N}}\overline g\right\rangle=\left\langle g,e^{2i(\theta_{g_N}-\theta_g)}g\right\rangle.
\end{eqnarray}
Since also  
\begin{eqnarray}\label{lims}
\lim_{N\rightarrow\infty}\|e^{2i(\theta_{g_N}-\theta_g)}-1\|=\lim_{N\rightarrow\infty}\|e^{2i\theta_{g_N}}-e^{2i\theta_g}\|=0\nonumber
\end{eqnarray} 
and again because $e^{2i\theta_{g_N}}$ is a Bohr function, it follows that $e^{2i(\theta_{g_N}-\theta_g)}\in U\subset B^2$ so the map $h\mapsto e^{2i(\theta_g-\theta_{g_N})}h$ is a bounded operator on $B^2$ by Lemma \ref{ugroup}. Thus, since norm convergence implies weak convergence, by (\ref{iprod2}) and (\ref{ser}) we have
\begin{eqnarray}\label{iprod4}
\|f\|^2=\lim_{N\rightarrow\infty}\Re\left\langle g,e^{2i(\theta_{g_N}-\theta_g)}g\right\rangle=\langle g,g\rangle=\|g\|^2.
\end{eqnarray} 
Lastly, writing 
\begin{eqnarray}\label{partialb}
g_{M}= \sum_{n\leq M} \langle g,e_{\lambda_n}\rangle e_{\lambda_n}\nonumber
\end{eqnarray}
(not necessarily the same as $g_N$ above) and using (\ref{iprod4}) and $\langle f,e_{\lambda}\rangle=\langle g,e_{\lambda}\rangle$, we have
\begin{eqnarray}
\|f-g_M\|^2&=&\|f\|^2+\|g_M\|^2-2\Re\langle f,\overline g_M\rangle\nonumber\\
&=&\|g\|^2-\|g_M\|^2\nonumber\\
&= &
\|g-g_M\|^2\leq \epsilon\nonumber\hspace{1cm}(M> M(\epsilon)),
\end{eqnarray}
showing that $f\in B^2$.

To complete the proof we show that if the maps $h\mapsto e^{i\theta_f}h$ and $h\mapsto e^{i\theta_g}h$ are also  bounded operators on $A^2$ (and therefore unitary) then $f\in B^2$. Indeed, we have
\begin{eqnarray}\label{unitary}
\|f\|^2=\langle |f|^2,1\rangle_{A^2}=\langle f^2e^{-2i\theta_f},1\rangle_{A^2}=\langle f^2,e^{2i\theta_f}\rangle_{A^2}
\end{eqnarray}
for $f^2\in A^2$ because the map $h\mapsto e^{i\theta_f}h$ ($h\in A^2$) is unitary. Then, from the definition of the $A^2$ inner product (\ref{innerprodH1}), since $\langle f^2,e_{\lambda}\rangle=\langle g^2,e_{\lambda}\rangle$, $e^{i\theta_f},e^{i\theta_g}\in U$ and $\langle e^{i\theta_f},e_{\lambda}\rangle=\langle e^{i\theta_g},e_{\lambda}\rangle$, (\ref{unitary}) is equal to
\begin{eqnarray}
\langle g^2,e^{2i\theta_f}\rangle_{A^2}=\langle g^2,e^{2i\theta_g}\rangle_{A^2}=\langle g^2e^{-2i\theta_g},1\rangle_{A^2}=\langle |g|^2,1\rangle_{A^2}=\|g\|^2\nonumber
\end{eqnarray}
which gives (\ref{iprod4}).
\end{proof}

\section{Proof of Theorems \ref{zetacns} and \ref{thzeta}}\label{zeta}

\subsection{Theorem \ref{zetacns}}This follows from  Proposition \ref{zo} if $f=\zeta^k$ satisfies the conditions of Proposition \ref{zo} assuming that $\mu(1/2)<1/2k$. Firstly, this implies that 
\begin{eqnarray}\label{zetacoefs}
\langle f^{j},e_{\lambda}\rangle = \left\{
                \begin{array}{ll}
                  \frac{d_{jk}(n)}{n^{\sigma}} \hspace{1.2cm}(\lambda=-\log n \hspace{0.3cm}(n\in\mathbb{N}))  \\
               0\hspace{1.9cm}(\textrm{otherwise}).
                \end{array}
              \right.
              \end{eqnarray}
for $j=1$ and $2$, by Proposition \ref{lem1}. Set $g_N=P_N^k$ and note that since $B^2$ is a Hilbert space there is a $g\in B^2$ such that 
\begin{eqnarray}\label{RF}
\langle g,e_{\lambda}\rangle=\lim_{N\rightarrow \infty}\langle g_N,e_{\lambda}\rangle = \left\{
                \begin{array}{ll}
                  \frac{d_{k}(n)}{n^{\sigma}} \hspace{1.2cm}(\lambda=-\log n \hspace{0.3cm}(n\in\mathbb{N}))  \\
               0\hspace{1.9cm}(\textrm{otherwise})
                \end{array}
              \right.
\end{eqnarray}
and $\|g-g_N\|<\epsilon$ for $N>N(\epsilon)$. Also, by (\ref{parseval}) we have 
\begin{eqnarray}\label{square}
\langle g^2,e_{\lambda}\rangle=\langle g,\overline ge_{\lambda}\rangle &=&
\sum_{\nu\in\mathbb{R}}\langle g,e_{\nu}\rangle \overline{\langle \overline ge_{\lambda},e_{\nu}\rangle}\nonumber\\
&=& \sum_{\nu\in\mathbb{R}}   \langle g,e_{\nu}\rangle \langle g,e_{\lambda-\nu}\rangle\nonumber\\
&=& \sum_{m=1}^{\infty}  \frac{d_k(m)}{m^{\sigma}} \langle g,e_{\lambda+\log m}\rangle
\end{eqnarray}
and, since $\lambda+\log m$ is the logarithm of the reciprocal of a natural number if and only if $\lambda=-\log n$ and $m|n$, (\ref{square}) is
\begin{eqnarray}\label{square2}
\langle g^2,e_{\lambda}\rangle= \left\{
                \begin{array}{ll}
                 \sum_{m|n}  \frac{d_k(m)}{m^{\sigma}} \frac{d_k(n/m)}{(n/m)^{\sigma}}=\frac{d_{2k}(n)}{n^{\sigma}}\hspace{1.5cm}(\lambda=-\log n \hspace{0.3cm}(n\in\mathbb{N}))  \\
               0\hspace{6cm}(\textrm{otherwise}).
                \end{array}
              \right.\nonumber\\
\end{eqnarray}
Thus by (\ref{zetacoefs}), (\ref{RF}) and (\ref{square2}), this $g\in B^2$ satisfies  $\langle f,e_{\lambda}\rangle=\langle g,e_{\lambda}\rangle$ and $\langle f^2,e_{\lambda}\rangle=\langle g^2,e_{\lambda}\rangle$ $(\lambda\in\mathbb{R})$. 

Now, since $\theta_{g_N}=\theta_{P^k_N}=\Im \log P^k_N$ $\pmod {(-\pi,\pi]}$
and 
\begin{eqnarray}
\Im \log P^k_N(\sigma+it)=-k\sum_{2\leq n\leq N^2}\frac{\Lambda_N(n)\sin\left(t\log n\right)}{n^{\sigma}\log n}\pmod {(-\pi,\pi]}\nonumber
\end{eqnarray}
by (\ref{Pdef}), it follows that 
\begin{eqnarray}
e^{i\theta_{g_N}(\sigma+it)}=\exp\left(-ik\sum_{2\leq n\leq N^2}\frac{\Lambda_N(n)\sin\left(t\log n\right)}{n^{\sigma}\log n}\right)\nonumber
\end{eqnarray}
is a Bohr function because $e^{iy}$ $(y\in\mathbb{R})$ is uniformly continuous and a uniformly continuous function of a uniformly convergent sequence is uniformly convergent. Since also 
\begin{eqnarray}\label{lims}
\lim_{N\rightarrow\infty}\left\|e^{i(\theta_{g_N}-\theta_f)}-1\right\|=\lim_{N\rightarrow\infty}\left\|e^{i\theta_{g_N}}-e^{i\theta_f}\right\|=0
\end{eqnarray} 
by Proposition \ref{prop4}, it follows that $e^{i\theta_f}\in U$. Moreover, since
\begin{eqnarray}\label{bohrf}
\ell_N=\sum_{2\leq n\leq N^2}\frac{\Lambda_N(n)}{n^{\sigma+i\cdot}\log n}\nonumber
\end{eqnarray}
is a Bohr function, we have $\ell_N\rightarrow \ell$  in $B^2$  as $N\rightarrow\infty$ in which the limit $\ell\in B^2$ is necessarily bounded except possibly on a null set, so that also the $\ell_N$ $(N\in\mathbb{N})$ are bounded except possibly on a null set. Therefore, since
\begin{eqnarray}
\left\|g-\exp\ell_N \right\|=\left\|g-g_N \right\|<\epsilon\hspace{1cm}(N>N(\epsilon)),\nonumber
\end{eqnarray}
we note that $g$ cannot vanish except possibly on a null set. As such, since
  \begin{eqnarray}
\left\||g|-|g_N| e^{i(\theta_{g_N}-\theta_g)}\right\|=\left\|g-g_N \right\|,\nonumber
\end{eqnarray}
it follows that 
\begin{eqnarray}
\textrm{Pr}\left(\left|\theta_{g_N}-\theta_g\pmod {(-\pi,\pi]}\right|\geq \epsilon\right)=0\hspace{1cm}(N>N(\epsilon))\nonumber
\end{eqnarray}
so $\lim_{N\rightarrow\infty}\|e^{i\theta_{g_N}}-e^{i\theta_g}\|=0$. In other words, $e^{i\theta_{g_N}}\rightarrow e^{i\theta_g}$ in $B^2$ as $N\rightarrow\infty$ so $e^{i\theta_g}\in U$ and $\|e^{i\theta_f}-e^{i\theta_g}\|=0$ by (\ref{lims}) and  the triangle inequality. Since the latter implies that  
$\langle e^{i\theta_f},e_{\lambda}\rangle=\langle e^{i\theta_g},e_{\lambda}\rangle$ $(\lambda\in\mathbb{R})$ by Cauchy-Schwarz, we conclude that $f=\zeta^k$ satisfies the conditions of Proposition \ref{zo}  provided  that $\mu(1/2)<1/2k$.

\subsection{Theorem \ref{thzeta}}If $\sigma_k>1/2$ we have
\begin{eqnarray}\label{inverses}
\limsup_{T\rightarrow\infty}\frac{M_k(\sigma,T)}{T}=\infty\hspace{1cm}(1/2\leq \sigma< \sigma_k).
\end{eqnarray}
Using the identity $1-\cos(2x)=2\sin^2(x)$, we then have 
\begin{eqnarray}\label{ebound}
M_k(\sigma,T)-\Re\int_{0}^T\zeta^{2k}e^{-2ik\theta_{P_N}}dt =2\int_{0}^T|\zeta|^{2k}\sin^2(k\theta_{Z_N})dt\nonumber
\end{eqnarray}
so that 
\begin{eqnarray}\label{liminf}
\limsup_{T\rightarrow\infty}\frac{\int_{0}^T|\zeta|^{2k}\sin^2(k\theta_{Z_N})dt}{M_k(\sigma,T)}=\frac{1}{2}-\liminf_{T\rightarrow\infty}\frac{\Re\frac{1}{T}\int_{0}^T\zeta^{2k}e^{-2ik\theta_{P_N}}dt}{\frac{2}{T}M_k(\sigma,T)}.
\end{eqnarray}
Since $\zeta^{2k}\in A^2$ if $\mu(1/2)<1/2k$, the numerator of the second term on the right hand side of (\ref{liminf}) converges to the limit $\langle \zeta^{2k},e^{2ik\theta_{P_N}}\rangle$ as $T\rightarrow\infty$. Therefore, 
\begin{eqnarray}\label{liminf2}
\liminf_{T\rightarrow\infty}\frac{\Re\frac{1}{T}\int_{0}^T\zeta^{2k}e^{-2ik\theta_{P_N}}dt}{\frac{2}{T}M_k(\sigma,T)} &\sim &\liminf_{T\rightarrow\infty}\frac{\Re \langle \zeta^{2k},e^{2ik\theta_{P_N}}\rangle}{\frac{2}{T}M_k(\sigma,T)}\nonumber\\
&=&\frac{\Re \langle \zeta^{2k},e^{2ik\theta_{P_N}}\rangle}{2}\left(\limsup_{T\rightarrow\infty}\frac{M_k(\sigma,T)}{T}\right)^{-1}=0\nonumber
\end{eqnarray}
by (\ref{inverses}).

On the other hand, if $\mu(1/2)<1/2k$, then a trivial modification of the proof of Proposition \ref{sig} shows that $\zeta^k(\sigma+it)\in B^2$ for $\sigma_k<\sigma<1$.
In particular, $\zeta^k(\sigma+it)$ is not CNS for  $\sigma_k<\sigma<1$, in which case we have
\begin{eqnarray}\label{probab}
\limsup_{T\rightarrow\infty}\frac{\int_{0}^T|\zeta|^{2k}\sin^2(k\theta_{Z_N})dt}{M_k(\sigma,T)}\leq \sup\{x:\Pr(\sin^2(k\theta_{Z_N})\geq x)>0\}<\epsilon\hspace{0.5cm}(N>N(\epsilon))\nonumber
\end{eqnarray}
 by (\ref{prob}) and Proposition \ref{prop4}.

\section{Proof of Propositions \ref{holder} and \ref{prop4}}\label{props}

\subsection{Proposition \ref{holder}}
Let $g$ be bounded, assume (\ref{meas}) and denote by $S^c$ the complement of $S$ in $\mathbb{R}^{+}$. Then for every null set $S$ we have 
\begin{eqnarray}\label{prob3}
\left|\int_{0}^Tf^2gdt \right|&=&\left|\int_{S^c\cap [0,T]}f^2gdt+\int_{S\cap [0,T]}f^2gdt\right|\nonumber\\
&\leq&\left|\int_{S^c\cap [0,T]}f^2gdt\right|+o\left(\int_{0}^T|f|^2dt\right)\nonumber\\
&\leq&(1+o(1))\sup_{t\in S^c\cap [0,T]}|g(t)|\int_{0}^T|f|^2dt.\nonumber
\end{eqnarray}
Since $S$ is an arbitrary null set, we minimise over those $S$ giving 
\begin{eqnarray}\label{prob4}
\limsup_{T\rightarrow\infty}\frac{\left|\int_{0}^Tf^2gdt\right|}{\int_{0}^T|f|^2dt}\leq \inf_{|S|=0}
\sup_{t\in S^c}|g(t)|= \sup\{x:\Pr(|g|\geq x)>0\}
\end{eqnarray}
which is (\ref{prob}). Conversely, assume (\ref{prob4}) and take $g=e^{-2i\theta_f}\chi_{S}$ where $\chi_S$ is the characteristic function of $S$. Then 
\begin{eqnarray}
\limsup_{T\rightarrow\infty}\frac{\int_{S\cap [0,T]}|f|^2dt}{\int_{0}^T|f|^2dt}=\limsup_{T\rightarrow\infty}\frac{\left|\int_{0}^Tf^2gdt\right|}{\int_{0}^T|f|^2dt}\leq \sup\{x:\Pr(|g|\geq x)>0\}=0\nonumber
\end{eqnarray}
which gives (\ref{meas}).

Now suppose that $\int_{0}^T|f|^2dt\gg T$. We have 
\begin{eqnarray}\label{prob5}
\left(\frac{\int_{S\cap [0,T]}|f|^2dt}{\int_{0}^T|f|^2dt}\right)^{1/2}
&\sim& \left| \left(\frac{\int_{S\cap [0,T]}|f|^2dt}{\int_{0}^T|f|^2dt}\right)^{1/2} - \left(\frac{\int_{S\cap [0,T]}|f_{N,S}|^2dt}{\int_{0}^T|f|^2dt}\right)^{1/2} \right|\nonumber\\
&\leq&  \left(\frac{\int_{S\cap [0,T]}|f-f_{N,S}|^2dt}{\int_{0}^T|f|^2dt}\right)^{1/2} \nonumber
\end{eqnarray}
so (\ref{ot}) implies (\ref{meas}). Conversely, (\ref{meas}) implies (\ref{ot}) by taking $f_{N,S}=0$. 

To complete the proof, we note that if $\int_{0}^T|f|^2dt\ll T$ then $f$ is bounded except possibly on a null set $S=S(f)$, so there is already a bounded function $f_{N,S^c}$ such that $|f-f_{N,S^c}|^2<\epsilon$ on $S^c$. Thus, taking 
\begin{eqnarray}
f_N=f_{N,S}\chi_{S}+f_{N,S^c}\chi_{S^c},\nonumber
\end{eqnarray}
we have
\begin{eqnarray}
\frac{1}{T}\int_{0}^T|f-f_N|^2 dt &=&\frac{1}{T}\int_{S\cap [0,T]}|f-f_{N,S}|^2dt+\frac{1}{T}\int_{S^c\cap [0,T]}|f-f_{N,S^c}|^2dt\nonumber\\
&\ll & \frac{\int_{S\cap [0,T]}|f-f_{N,S}|^2dt}{\int_{0}^T|f|^2dt}+\frac{\epsilon}{T}\textrm{meas}\left(S^c\cap [0,T]   \right)
\nonumber
\end{eqnarray}
and using (\ref{ot}) gives (\ref{sup}). Conversely,  (\ref{sup}) implies  (\ref{ot}) if $\int_{0}^T|f|^2dt\gg T$.

\subsection{Proof of Proposition \ref{prop4}}
For $X>2$ we write 
\begin{eqnarray}
\textrm{Arg} Z_X =\theta_{ Z_X}\pmod {(-\pi,\pi]}=\Im \log Z_X \pmod {(-\pi,\pi]}\nonumber
\end{eqnarray} 
and note that the result will follow if it can be shown that there is a null set $R=R(\sigma,T)$ such that for every fixed $\sigma>1/2+\delta$ and $\epsilon>0$ there is an $N(\epsilon)$ such that   
\begin{eqnarray}\label{toprove}
\limsup_{T\rightarrow\infty}\frac{1}{T}\int_{[0,T]\setminus R}|\log Z_N (\sigma+it)|dt< \epsilon \hspace{1cm}(N>N(\epsilon)).
\end{eqnarray}
Let $N(\sigma, T)$ denote the number of zeros $\rho=\beta+i\gamma$ with $\beta\geq\sigma$ and $|\gamma|\leq T$ and recall the well-known result that  for every fixed $\sigma>1/2$ there is an $\alpha(\sigma)<1$ such that $N(\sigma, T)\ll T^{\alpha(\sigma)}$. Accordingly, for $\delta>0$ fixed and  $1/2+\delta\leq \sigma<1$  note that the number of zeros with $\beta\geq 1/4 + \sigma/2$ and $|\gamma|\leq T$ is $\ll T^{\alpha(1/4+\sigma/2)}$ for some $\alpha(1/4+\sigma/2)<1$. Thus if $\Delta=o\left(T^{1-\alpha (1/4+\sigma/2)}\right)$ we may omit every interval of the line segment $[\sigma,\sigma+iT]$ on which $|t-\gamma|<\Delta$ for some $\rho=\beta+i\gamma$ with $\beta\geq 1/4 + \sigma/2$ while ensuring that the union $S(\sigma, T)$ of the remaining intervals of the line segment have 1-dimensional measure $\sim T$ and, in particular, that if $|t-\gamma|<\Delta$ for some $s\in S$ and $\rho=\beta+i\gamma$ then 
\begin{eqnarray}\label{realpartbound}
\sigma-\beta >\sigma-1/4-\sigma/2 \geq \delta/2.
\end{eqnarray}

Firstly we will show that for a suitable choice of $X=X(T)$ the integral
\begin{eqnarray}\label{bi} 
\frac{-i}{T}\int_{S}|\log Z_X (s)|ds\ll \frac{1}{\log^AT}\hspace{1cm}(A\geq 0)
\end{eqnarray}
Then we will use an approximation argument to show that this implies (\ref{toprove}). To this end, recall the formula 
\begin{eqnarray}\label{goneks}
\log Z_X(s)=\sum_{\rho}F_2\left((s-\rho)\log X\right)+O\left(\frac{X^{2-2\sigma}}{(1+t^2)\log^{2} X}\right)
\end{eqnarray}
where  
\begin{eqnarray}
F_2(z)=\int_{z}^{\infty}\frac{e^{-2w}-e^{-w}}{w^2}dw\nonumber
\end{eqnarray}
due to Gonek (\cite{G}, p. 10), and note that the contribution of the error term in (\ref{goneks}) to the integral (\ref{bi}) is 
\begin{eqnarray}\label{errorint}
\ll \frac{X}{T\log^2 X}.
\end{eqnarray} Since
\begin{eqnarray}\label{integrand}
\int_{(s-\rho)\log X}^{\infty}\frac{e^{-w}dw}{w^2} =\int_{0}^{\infty}\frac{X^{-w-s+\rho}dw}{(w+s-\rho)^2\log X}
\ll \frac{X^{\beta-\sigma}} {|s-\rho|^2\log X},\nonumber
\end{eqnarray}
we write
\begin{eqnarray}
\sum_{\rho}&=&\sum_{\rho:|t-\gamma|<\Delta}+\sum_{\rho:\Delta\leq |t-\gamma|<T}+\sum_{\rho:|t-\gamma|\geq T}\nonumber
\end{eqnarray} 
and observe that on $S$ the sum in (\ref{goneks}) is
\begin{eqnarray}\label{zerosumbound}
&\ll & \frac{X^{-\delta/2}}{\delta^2\log X}\sum_{\substack{|t-\gamma|<\Delta\\\sigma-\beta\geq \delta/2}}1
+\frac{X^{-\delta/2}}{\log X}
\sum_{\substack{\Delta\leq |\gamma-t|< T\\\sigma-\beta\geq \delta/2}}\frac{1}{(\gamma-t)^2} \nonumber\\
&+&\frac{X}{\log X}
\sum_{\substack{\Delta\leq |\gamma-t|< T\\\sigma-\beta < \delta/2}}\frac{1}{(\gamma-t)^2} 
+\frac{X}{\log X}\sum_{|\gamma-t|\geq T}\frac{1 }  {(\gamma-t)^2}.
\end{eqnarray}
Recalling now that that the total number of zeros  $\rho=\beta+i\gamma$ with $0<\gamma< t$ is
\begin{eqnarray}
N(t)=\frac{t}{2\pi}\log\frac{t}{2\pi}-\frac{t}{2\pi}+\frac{7}{8}+O(\log t),\nonumber
\end{eqnarray}
it may be easily deduced (see for instance Gonek (\cite{G}, pp. 8-10)) that
\begin{eqnarray}\label{zerocont}
\sum_{\rho:|\gamma-t|< \Delta}1\ll \Delta \log t\hspace{0.3cm}\textrm{ and }\hspace{0.3cm}\sum_{\rho:|\gamma-t|\geq \Delta}\frac{1}{(\gamma-t)^2}\ll \frac{\log t}{\Delta}.
\end{eqnarray}
Estimating the first, second and fourth summations in (\ref{zerosumbound}) using (\ref{zerocont}) we find that it is 
\begin{eqnarray}\label{zerosumbound2}
\ll_{} \frac{X}{\log X}
\sum_{\substack{\Delta\leq |\gamma-t|< T\\\sigma-\beta < \delta/2}}\frac{1}{(\gamma-t)^2}+\frac{\log t}{\log X}
\left(X^{-\delta/2}\left(\frac{\Delta}{\delta^2}+\frac{1}{\Delta}\right)+\frac{X}{T}  \right).\nonumber
\end{eqnarray}
Now if $X(T)=\log^{(2+2A)/\delta} T$ for some fixed $A\geq 0$, the above is 
\begin{eqnarray}\label{zerosumbound3}
&\ll & \frac{\delta\log^{(2+2A)/\delta} T}{\log \log T} \sum_{\substack{\Delta\leq |\gamma-t|< T\\\sigma-\beta < \delta/2}} \frac{1}{(t-\gamma)^2}+\frac{1}{\log ^AT\log \log T}
\left(\frac{\Delta}{\delta}+\frac{\delta}{\Delta} \right)\nonumber
\end{eqnarray}
so, given also the estimate (\ref{errorint}) with this choice of $X$, the integral (\ref{bi}) is
\begin{eqnarray}\label{reduction}
&\ll &\frac{\delta\log^{(2+2A)/\delta} T}{T\log \log T} \sum_{\substack{  |\gamma|< 2T\\  \sigma-\beta < \delta/2}}  \int_{\{0<t<T:\Delta<|t-\gamma|<T\}} \frac{dt}{(t-\gamma)^2}
\nonumber\\
&+&\frac{1}{\log ^AT\log \log T}
\left(\frac{\Delta}{\delta}+\frac{\delta}{\Delta} \right)
\end{eqnarray}
in which the integral is bounded by 
\begin{eqnarray}
\int_{\Delta<t<3T} \frac{dt}{t^2}\ll \frac{1}{\Delta}+\frac{1}{T}\nonumber
\end{eqnarray}
so (\ref{reduction}) is
\begin{eqnarray}\label{reduction2}
&\ll &\frac{\delta T^{\alpha(1/2+\delta/2)-1}\log^{(2+2A)/\delta} T}{\Delta \log \log T}+\frac{1}{\log ^AT\log \log T}
\left(\frac{\Delta}{\delta}+\frac{\delta}{\Delta} \right),\nonumber
\end{eqnarray}
where we have used the zero density estimate $\ll T^{\alpha(1/2+\delta/2)}$  for the number of zeros with $|\gamma|<2T$ and $\beta>\sigma-\delta/2\geq 1/2+\delta/2$. Since this exponent is strictly less than one for any fixed $\delta>0$ and $\Delta$ was arbitrary, this proves (\ref{bi}) for this choice of $X$. \\

We now use an approximation argument to show that the above conclusion implies the proposition. Using the mean value theorem for Dirichlet polynomials with $X$ as above, $N<\log^{(1+A)/\delta} T$ and $\sigma>1/2+\delta$ we have 
\begin{eqnarray}\label{dyadicsums}
\int_{T}^{2T}\left |\log \frac{P_X}{P_N}(\sigma+it)\right|^2dt &\ll& T \sum_{n>N}\frac{\left |\Lambda_X(n)-\Lambda_N(n)\right|^2}{n^{2\sigma}\log^2 n} \nonumber\\
&\ll& TN^{1-2\sigma+\delta}\nonumber\\
&=&TN^{-\delta}.
\end{eqnarray}
Replacing $T$ by $T/2$, $T/4$,... in (\ref{dyadicsums}) and summing, we obtain
\begin{eqnarray}\label{pxpnbound}
\frac{1}{T}\int_{0}^{T}\left |\log \frac{P_X}{P_N}(\sigma+it)\right|^2dt \ll N^{-\delta}.
\end{eqnarray}
Lastly, using the triangle inequality we have 
\begin{eqnarray}\label{proved}
\frac{-i}{T}\int_{S}|\log Z_N (s)|ds &\leq &\frac{-i}{T}\int_{S}|\log Z_X (s)|ds+\frac{1}{T}\int_{0}^{T}\left |\log \frac{P_X}{P_N}(\sigma+it)\right|dt\nonumber\\
&\ll& \frac{1}{\log^AT}+N^{-\delta/2}\nonumber
\end{eqnarray}
by (\ref{bi}) and (\ref{pxpnbound}), which proves (\ref{toprove}).

\section{proofs of the auxiliary results}\label{auxs}

\subsection{Proposition \ref{sig}}
Clearly $\zeta^k\in B^2$ implies that $\sigma_k=1/2$, so we assume $\sigma_k=1/2$ and prove the converse statement. 
It follows from $\mu(1/2)< 1/2k$ and Proposition \ref{lem1} that 
\begin{eqnarray}
\langle \zeta^k,e_{\lambda}\rangle = \left\{
                \begin{array}{ll}
                  \frac{d_{k}(n)}{n^{\sigma}} \hspace{1.25cm}(\lambda=-\log n \hspace{0.3cm}(n\in\mathbb{N}))  \\
               0\hspace{1.83cm}(\textrm{otherwise})
                \end{array}
              \right.\nonumber
              \end{eqnarray}
so, if we set $f=\zeta^k$ and 
\begin{eqnarray}
f_{N}=\sum_{n\leq N}\frac{d_{k}(n)}{n^{\sigma}}n^{-i\cdot},\nonumber
\end{eqnarray}
then 
\begin{eqnarray}\label{beseq}
\|f-f_N\|^2= \|f\|^2+\|f_N\|^2-2\Re\langle f,\overline f_N    \rangle =\|f\|^2-\|f_N\|^2
= \sum_{n> N} \frac{d^2_{k}(n)}{n^{2\sigma}}
\end{eqnarray}
by (\ref{asy3}). Since (\ref{beseq}) is clearly $<\epsilon$ for $N> N(\epsilon)$, we see that $\sigma_k=1/2\Rightarrow \zeta^k\in B^2$.

\subsection{Proposition \ref{lem1}}
Fix $\delta>0$. Denote by $R$ the rectangle with vertices $[\sigma+iT,1+\delta+iT,1+\delta+i,\sigma+i]$ and consider the integral 
\begin{eqnarray}\label{inter}
\frac{1}{2iT}\int_{R}\zeta^{k}(s)e^{\lambda(2\sigma-s)}ds.
\end{eqnarray}
Assuming that $\mu(1/2)<1/k$, there is an $A>0$ such that the sum of the horizontal segments is
\begin{eqnarray}\label{ppo}
\ll \frac{1}{T} \int_{\sigma}^{1+\delta}|\zeta(r+iT)|^{k}e^{\lambda(2\sigma-r)}dr+\frac{e^{\lambda \sigma}}{T}\ll \frac{e^{\lambda \sigma}}{T^{A}}
\end{eqnarray}
and the sum of the vertical segments is
\begin{eqnarray}\label{vgg}
&&\frac{e^{2\lambda\sigma}}{T}\int_{1}^{T}\zeta^{k}(1+\delta+it)e^{-\lambda(1+\delta+it)}dt-\frac{e^{\lambda\sigma}}{T}\int_{1}^{T}\zeta^{k}(\sigma+it)e^{-i\lambda t}dt\nonumber\\
&=&\sum_{n=1}^{\infty} \frac{d_{k}(n)}{n^{1+\delta}}\frac{e^{\lambda(2\sigma-1-\delta)}}{T}\int_{1}^{T}e^{-i(\lambda+\log n) t}dt -\frac{e^{\lambda\sigma}}{T}\int_{1}^{T}\zeta^{k}(\sigma+it)e^{-i\lambda t}dt
\end{eqnarray}
because the Dirichlet series above is absolutely convergent. By (\ref{inter}), (\ref{ppo}) and (\ref{vgg}) we see that 
\begin{eqnarray}
\frac{1}{T}\int_{1}^{T}\zeta^{k}(\sigma+it)e^{-i\lambda t}dt&=&e^{\lambda(\sigma-1-\delta)}\sum_{n=1}^{\infty} \frac{d_{k}(n)}{n^{1+\delta}}\frac{1}{T}\int_{1}^{T}e^{-i(\lambda+\log n) t}dt+O(T^{-A})\nonumber\\
&=&\left\{
                \begin{array}{ll}
                  \frac{d_{k}(n)}{n^{\sigma}} +o\left(1\right)\hspace{1.25cm}(\lambda=-\log n)  \\
               o\left(1\right)\hspace{2.55cm}(\textrm{otherwise}).
                \end{array}
              \right.\nonumber
\end{eqnarray}

\subsection{Lemma \ref{mainlemma}} $A^2$ is an inner product space by definition. Since $B^2$ is a dense subset (in the $A^2$ metric, of course), the completeness of $A^2$ is inherited from that of $B^2$ and for every $f\in A^2$ we then have $f=\phi+g$ pointwise for some $\phi\in E^{\perp}$ and $g\in B^2$.

\subsection{Lemma \ref{ugroup}}
Let $u\in B^2$ and $\nu\in\mathbb{R}$ be fixed. Since the set $\{e_{\lambda}\}_{\lambda\in\mathbb{R}}$ is a complete orthonormal set in $B^2$, for every $\epsilon>0$ there is an integer $M(u,\epsilon)$ and a sequence of real numbers $\lambda_1,\cdots,\lambda_{M}$ such that  
\begin{eqnarray}\label{uexp}
\left\|\overline u -\sum_{m\leq M}\left\langle \overline u, e_{\lambda_m-\nu}\right\rangle e_{\lambda_m-\nu}\right\|<\epsilon  \hspace{1cm}(M>M(u,\epsilon)).
\end{eqnarray}
If $|u|=1$, then it is immediate from (\ref{uexp}) that

\begin{eqnarray}\label{uexp2}
\left\|e_{\nu}-\sum_{m\leq M}\left\langle e_{\nu},ue_{\lambda_m}\right\rangle ue_{\lambda_m}\right\|<\epsilon \hspace{1cm}(M>M(u,\epsilon)).
\end{eqnarray}

For each $\lambda\in\mathbb{R}$ we write $u_{\lambda}=ue_{\lambda}$ and note that the set $\{u_{\lambda}\}_{\lambda\in\mathbb{R}}$ is an orthonormal set in $B^2$. Now suppose that the set $\{u_{\lambda}\}_{\lambda\in\mathbb{R}}$ is not dense in $B^2$, and let $V$ be the proper subspace spanned by $\{u_{\lambda}\}_{\lambda\in\mathbb{R}}$. Since $\{e_{\nu}\}_{\nu\in\mathbb{R}}$ is a complete orthonormal set there is an $e_{\nu_0}\in V^{\perp}$, which contradicts (\ref{uexp2}) so $\{u_{\lambda}\}_{\lambda\in\mathbb{R}}$ is dense in $B^2$. Since a dense orthonormal set in a Hilbert space is complete, it follows that the set $\{ u_{\lambda}\}_{\lambda\in\mathbb{R}}$ is a complete orthonormal set in $B^2$ and by Parseval's theorem we have  
\begin{eqnarray}\label{par2}
\sum_{\lambda\in\mathbb{R}}|\langle \overline uf, e_{\lambda}\rangle |^2=\sum_{\lambda\in\mathbb{R}}|\langle f, u_{\lambda}\rangle |^2  = \|f\|^2\hspace{1cm}     (f\in B^2).\nonumber
\end{eqnarray}
Thus the transformations $M_u$ are continuous on $B^2$ and therefore unitary. 

Lastly, to see that $U$ is a multiplicative subgroup, note that if $u,v\in U$ and $uv=w$ then 
\begin{eqnarray}\label{uexp1}
\left\|w -\sum_{m\leq M}\left\langle w, v_{\lambda_m}\right\rangle v_{\lambda_m}\right\|=\left\|u -\sum_{m\leq M}\left\langle u, e_{\lambda_m}\right\rangle e_{\lambda_m}\right\|<\epsilon\nonumber
\end{eqnarray}
which implies $w\in U$ because $\{ v_{\lambda}\}_{\lambda\in\mathbb{R}}$ is a complete orthonormal set in $B^2$.\\

\paragraph{\bf{Acknowledgement}} I would like to thank Julio Andrade, Roger Heath-Brown and Christopher Hughes for their comments and suggestions on this work. I am also grateful to be supported by the Faculty of Environment, Science and Economy at the University of Exeter.

\noindent\emph{Email address}:\texttt{ks614@exeter.ac.uk}

\end{document}